\documentclass[10pt,a4paper,3p]{elsarticle}

\usepackage{graphicx,amssymb,amsthm}
\usepackage[IL2]{fontenc}
\usepackage{enumitem}
\usepackage{float}
\newtheorem{lemma}{Lemma}
\newtheorem{theorem}{Theorem}
\newtheorem{corollary}{Corollary}
\newtheorem{conjecture}{Conjecture}

\newcommand{\subqed}{{\ifmmode q.e.d. \else\unskip\nobreak\hfil
\penalty50\quad\null\nobreak\hfill $\blacksquare$ \parfillskip=0pt
\finalhyphendemerits=0\par\fi}}

\newcommand{\ml}{l\kern-0.55mm\char39\kern-0.3mm}

\begin{document}

\begin{frontmatter}

\title{On the cyclic coloring conjecture}


\author{Stanislav Jendro{\ml}}
\ead{stanislav.jendrol@upjs.sk}

\author{Roman Sot\'{a}k}
\ead{roman.sotak@upjs.sk}

\address{Institute of Mathematics, P.~J.~\v Saf\' arik University, Jesenn\'a 5, 040 01 Ko\v sice, Slovakia}

\begin{abstract}
A  cyclic coloring of a plane graph $G$ is a coloring of its vertices such that vertices incident with the same face have distinct colors. The minimum number of colors in a cyclic coloring of a plane graph $G$ is its cyclic chromatic number 
$\chi_c(G)$. Let $\Delta^*(G)$ be the maximum face degree of a graph $G$.

In this note we show that to prove the Cyclic Coloring Conjecture of Borodin from 1984, saying that every connected plane graph $G$ has $\chi_c(G) \leq \lfloor \frac{3}{2}\Delta^*(G)\rfloor$, it is enough to do it for subdivisions of simple $3$-connected plane graphs. 

We have discovered four new different upper bounds on $\chi_c(G)$ for graphs $G$ from this restricted family; three bounds of them are tight. As corollaries, we have shown that the conjecture holds for subdivisions of plane triangulations, simple $3$-connected plane quadrangulations, and simple $3$-connected plane pentagulations with an even maximum face degree, for regular subdivisions of simple $3$-connected plane graphs of maximum degree at least 10, and for subdivisions of simple $3$-connected plane graphs having the maximum face degree large enough in comparison with the number of vertices of their longest paths consisting only of vertices of degree two.

\end{abstract}


\end{frontmatter}

\section{Introduction}
Throughout this note we use graph theory terminology according to the books \cite{MohTho-01} and \cite{West-01}. However, we recall the most frequent notions. In this note,  $G$ is a connected plane graph with vertex set $V(G)$, edge set $E(G)$, and face set $F(G)$. In what follows, $G$ can have multiple edges but no loops, while a {\it{simple}} graph has no multiple edges. The {\it{degree}} of a vertex $v$, denoted by $\deg_G(v)$, is the number of edges incident with $v$. The {\it{degree}} of a face $f$, denoted by $\deg_G(f)$, is the number of vertices incident with $f$. A $k$-face is any face of degree $k$. We use $\Delta(G)$ and $\Delta^*(G)$ to denote the {\it{maximum vertex degree}} and {\it{maximum face degree}} of $G$, respectively. In a graph $G$, a {\it{subdivision}} of an edge $uv$ is the operation of replacing $uv$ with a paths $u,w,v$ through a new vertex $w$. A {\it{subdivision}} of a graph $G$ is a graph obtained by a sequence of subdivisions of edges of $G$.

For a cycle $C$ (in a plane graph $G$) we denote the set of vertices and edges  of $G$ lying inside $C$ and outside $C$ by  
$\mathrm {int}{_G(C)}$ and $\mathrm {ext}{_G(C)}$, respectively. We say that $C$ is a {\it{separating cycle}} if both
$\mathrm {int}{_G(C)}$ and $\mathrm {ext}{_G(C)}$ are not empty.

Let $G$ be a connected plane graph and $C_n$ be a separating cycle of length $n$ in $G$. Let $G_1$ (resp. $G_2$) be a plane graph obtained from $G$ by deleting the exterior (resp. the interior) of $C_n$. Observe that $\Delta^*(G_i) \leq \max\{\Delta^*(G), n\}$, $i\in\{1, 2\}$, where  $n$ is the degree of the face $g_i$ of $G_i$ whose boundary is the cycle $C_n$. The face $g_i$ is not a face of $G$ while any  other face of $G$ is present in $G_1$ or in $G_2$.

A {\it{cyclic coloring}} of a plane graph is a vertex coloring such that any two different {\it{cyclically adjacent}} vertices, i.e. vertices incident with the same face, receive distinct colors. The minimum number of colors needed for a cyclic coloring of a graph $G$, the {\it{cyclic chromatic number}}, is denoted by $\chi_c(G)$. This concept was introduced by Ore and Plummer \cite{OrePlu-69}.

It is obvious that any cyclic coloring of a connected plane graph $G$ requires at least $\Delta^*(G)$ colors. As concerns the upper bound, Borodin \cite{Bor-84} has conjectured (see also Conjecture 6.1 in \cite{Bor-13}, or Section 2.5 in \cite{JenTof-95}):

\begin{conjecture}{\rm\cite{Bor-84}\label{CCC:1}}
If $G$ is a connected plane graph  with maximum face degree $\Delta^*(G)$, then   
$$\chi_c(G) \leq \lfloor \frac{3}{2}\Delta^*(G)\rfloor.$$
\end{conjecture}

As shown in \cite{Bor-84}, the upper bound in Conjecture \ref{CCC:1} whenever true is best possible.

This conjecture is known as the Cyclic Coloring Conjecture, see e.g. \cite{HavKra-10}, \cite{HavSer-08}, or \cite{KraMad-05}.

From the Four Color Theorem  \cite{ApH-77a} and  \cite{ApH-77b} it follows that $\chi_c(G) \leq 4$ if $\Delta^*(G) = 3$. From Borodin's proof of the Ringel's conjecture \cite{Bor-84} (see also \cite{Bor-95}) we have $\chi_c(G) \leq 6$ if $\Delta^*(G) \leq 4$. Both these bounds are tight. Hebdige and Kr\'{a}{\ml} proved this conjecture for $\Delta^*(G) = 6$, see \cite{HebKra-17}. So in this note we shall study only the cases when $\Delta^*(G) \geq 5$. Ore and Plummer \cite{OrePlu-69} proved that $\chi_c(G) \leq 2\Delta^*(G)$, which was improved to $\lfloor \frac{9}{5}\Delta^*(G)\rfloor$ by Borodin, Sanders, and Zhao  \cite{BorSZ-99}, and to $\lceil \frac{5}{3}\Delta^*(G)\rceil$ by Sanders and Zhao \cite{SanZha-01}. Moreover, Borodin at al. \cite{BorSZ-99} proved that $\chi_c(G) \leq 8$ for $\Delta^*(G) = 5$. Amini, Esperet, and van den Heuvel \cite{AEH-13} proved that the conjecture holds asymptotically in the following sense: for every $\epsilon > 0$, there exists $\Delta_{\epsilon}$ such that every plane graph with maximum face degree $\Delta^* \geq \Delta_{\epsilon}$ has a cyclic coloring with at most $(\frac{3}{2}+\epsilon)\Delta^*$ colors. 

A beautiful result was obtained by Borodin, Broersma, Glebov, and van den Heuvel \cite{BBGH-07} using the following parameter 
$k^* = k^*(G)$, which denotes the maximum number of vertices that two faces of $G$ can have in common.

\begin{theorem}{\rm\cite{BBGH-07}}\label{thm:k^*}
For every connected plane graph $G$ with $\Delta^*(G) \geq 5$ it holds

$$\chi_c(G) \leq \max\{\Delta^*(G) + 3k^* + 2,\Delta^*(G) + 14\}.$$

\end{theorem}

They also posed the following:
\begin{conjecture}{\rm\cite{BBGH-07}\label{BBC:1}}
Every plane graph $G$ with $\Delta^*(G)$ and $k^*$ large enough has a cyclic coloring with  $$\Delta^*(G) + k^*$$ colors.
\end{conjecture}

Better bounds are known for simple $3$-connected plane graphs. The first bound for such graphs, $\chi_c(G) \leq \Delta^*(G) + 9$, was obtained by Plummer and Toft \cite{PluTof-87}. They also conjectured that any simple $3$-connected plane graph has $\chi_c(G) \leq \Delta^*(G) + 2$. The best presently known upper bounds for the cyclic chromatic number of simple $3$-connected plane graphs are:  $\chi_c(G) \leq \Delta^*(G) + 1$ for $\Delta^*(G) \geq 122$ by Borodin and Woodall \cite{BorWo-99} and
for $\Delta^*(G) \geq 60$ by Enomoto, Hor\v n\'{a}k, and Jendro{\ml} \cite{EnoHorJen-01}, $\chi_c(G) \leq \Delta^*(G) + 2$ for $\Delta^*(G) \geq 24$ by Hor\v n\'{a}k and Jendro{\ml} \cite{HorJen-99}, for $\Delta^*(G) \geq 18$ by Hor\v n\'{a}k and Zl\'{a}malov\'{a} \cite{HorZla-10}, for $\Delta^*(G) \geq 16$ by Dvo\v r\'{a}k, Hebdige, Hl\'{a}sek, Kr\'{a}{\ml}, and Noel \cite{DHHK-17}, for $\delta(G) = 4$ and $\Delta^*(G) \geq 6$, or 
$\delta(G) = 5$ by Zl\'{a}malov\'{a} \cite{Zla-10}, and for $G$ being locally connected by Kriesell \cite{Kri:06}. Azarija, Erman, Kr\'{a}{\ml}, Krnc, and Stacho \cite{AEK-12} proved that for every plane graph $G$, in which any two faces of degree at least four are vertex disjoint, holds  $\chi_c(G) \leq \Delta^*(G) + 1$. 
Havet, Sereni, and \v Skrekovski \cite{HavSer-08} proved that  $\chi_c(G) \leq 11$ for 
$\Delta^*(G) = 7$. The best presently known general upper bound, $\chi_c(G) \leq \Delta^*(G) + 5$, is by Enomoto and Hor\v n\'{a}k \cite{EnoHor-09}. Summarizing we have:

\begin{theorem}\label{thm:+r}
If $G$ is a simple $3$-connected plane graph, then 
$\chi_c(G) \leq \Delta^*(G) + r$, 
where
\begin{enumerate}
\item $r=1$ if $\Delta^*(G) \geq 60$, or $\Delta^*(G) = 3$, or any two faces of $G$ of degree at least four are vertex disjoint, 
\item $r=2$ if $\Delta^*(G) \geq 16$, or $\Delta^*(G) = 4$, or $\delta(G) = 4$ and $\Delta^*(G) \geq 6$, or $\delta(G) = 5$, or $G$ is locally connected,
\item $r=3$ if $5 \leq \Delta^*(G) \leq 6$, 
\item $r=4$ if $\Delta^*(G) = 7$, and
\item $r=5$ in the remaining cases.
\end{enumerate}
\end{theorem}

In this paper we show that to prove Conjecture \ref{CCC:1} it is enough to do it for subdivisions of simple $3$-connected plane graphs. Next we have obtained four different upper bounds on $\chi_c(G)$ for graphs from this family; three of them are tight.  As corollaries  we show that Conjecture \ref{CCC:1} holds for large maximum face degree subdivisions of simple $3$-connected plane graphs, and for four wide families of plane graphs without restrictions on maximum face degrees.



\section{Definitions and some remarks about the structure of plane graphs}\label{s:equable}

Throughout this paper $P_m$ will denote a path of length $m$. By $t = t(G)$ we will denote the number of vertices of a longest path in $G$ all vertices of which have degree 2.

For a $2$-connected plane graph $G$ other than cycle the {\it{reduction}} $R(G)$ of $G$ is a graph obtained from $G$ by replacing all maximal $u,v$-paths whose all interior vertices are of degree 2 with the edges $uv$. More precisely,
$V(R(G)) = \{v \in V(G)| \deg_G(v)\geq 3\} = V_{\geq 3}(G)$ and $E(R(G)) = P(G)$ which is the set of all maximal paths of $G$ all interior vertices of which have degree 2 and whose ends are vertices of degree at least 3.

Observe, that $|F(G)| = |F(R(G))|$ and that there is a one-to-one correspondence between sets of faces of $F(G)$ and $F(R(G))$, between the set of edges $E(R(G))$ and the set of paths $P(G)$, and between the sets of vertices $V(R(G))$ and $V_{\geq 3}(G)$. Clearly, $R(G)$ is a $2$-connected plane graph with minimum degree 3.

Observe, that $R(G)$ has  at least one of the following properties:
\begin{enumerate}
\item
It is a simple $3$-connected graph.
\item
It contains a 2-face [u,v].
\item
It has a vertex-cut $\{u, v\}$ and a component $K$ of $R(G)\setminus \{u, v\}$ such that a $2$-connected subgraph $H_R(u,v)$ on the vertex set $V(K)\cup\{u,v\}$ has the following structure: All vertices and edges of $H_R(u,v)$ lie on or in the interior (resp. in the exterior) of a separating cycle $C'$ determined by two internally vertex disjoint $u,v$-paths $P'$ and $P''$ whose all internal (resp. external) vertices are from the set $V(K)$.
\end{enumerate}

Note, that it is enough to choose a vertex-cut $\{u, v\}$ of  minimal number of vertices of $V(K)$. Then $H_R(u,v)$ is the subgraph of $R$ consisting of the vertex set $V(K)\cup\{u,v\}$ and the edge set containing $E(K)$ and all edges of $R$  between the vertices of $V(K)$ and the vertices of $\{u, v\}$ (and no edge $uv$). From the minimality of $K$ it follows that $H_R(u,v)$ is $2$-connected. The boundary cycle of $H_R(u,v)$ containing both $u$ and $v$ is chosen as a desired cycle  $C'$. 

\section{A structural Theorem}

\begin{theorem}\label{thm:2}
If $G$ is a $2$-connected plane graph with maximum face degree $\Delta^*(G)$ and at least four faces, then $G$ is either 
a subdivision of a simple $3$-connected plane graph or contains a separating cycle $C_n$ with 
$n \leq\Delta^*(G)$.  
\end{theorem}

\begin{proof}
Let $\Delta^* = \Delta^*(G)$. It is easy to see that the theorem holds if $G$ has exactly four faces or it contains 
a $2$-face. Next we distinguish two cases depending on the structure of the reduction $R = R(G)$.

 If $G$ is a subdivision of a simple $3$-connected plane graph (i.e. $R$ is a simple $3$-connected plane graph), then there is nothing to prove.

Otherwise the reduction $R$ contains either a $2$-face [u,v] or has a suitable vertex-cut \{u, v\}. Consider $H_G(u,v)$, the subgraph of 
$G$ corresponding to the $2$-face [u,v] or to the subgraph $H_R(u,v)$ of $R$. 

Let the cycle $C_{a+b}$, bounding $H_G(u,v)$ in $G$, be defined by the $u,v$-paths $P_a$ and $P_b$  corresponding to the edges 
$uv$ of the $2$-face $[u,v]$ or to the $u,v$-paths $P'$ and $P''$ of the subgraph $H_R(u,v)$, respectively, of $R$.  Let $f_1$ and $f_2$ be the faces of $G$ in the exterior of $C_{a+b}$ defined by the paths $P_a$ and $P_b$, respectively. Let $P_c$ and $P_d$ be the other $u,v$-paths bounded the faces $f_1$ and $f_2$, respectively. This is always possible because in every $2$-connected plane graph any face is bounded by a cycle. Observe that $\deg_G(f_1) = a + c \leq \Delta^*$ and $\deg_G(f_2) = b + d \leq \Delta^*$. Let, w.l.o.g., $b\geq a$. Then $a + d \leq \Delta^*$  and there is, in $G$, a separating cycle $C_n$ of length at most $\Delta^*$, namely $C_n = C_{a+d}$.

\end{proof}


\section{Two Lemmas}

\begin{lemma}{\label{le:1}}
If $G$ is a $2$-connected plane graph with exactly three faces and maximum face degree  $\Delta^*(G)$, then
$$\chi_c(G) \leq \lfloor \frac{3}{2}\Delta^*(G)\rfloor.$$
\end{lemma}

\begin{proof}
The graph $G$ consists of three edge disjoint paths $P_a$, $P_b$, and $P_c$ joining two vertices $u$ and $v$. Because for any two vertices there is, in $G$, a common face incident with both of them, we have $\chi_c(G) = a+b+c-1$. Observe that 
$a+b\leq \Delta^*(G)$, $a+c\leq \Delta^*(G)$, and $c+b\leq \Delta^*(G)$. This implies 

$$2(a+b+c-1) + 2 = 2\chi_c(G) + 2 \leq 3\Delta^*(G)$$

which immediately gives the statement of the lemma.

\end{proof}

The construction of the following lemma first appears in \cite{Bor-84}.

\begin{lemma}{\label{le:2}}
For every $t \geq 0$ there exists a $2$-connected plane graph $H$ of maximum face degree $\Delta^*(H)$ with the reduction $R(H)$ being simple $3$-connected plane graph, $t = t(H)$, and
$$\chi_c(H) = \Delta^*(H) + t + 2 = \lfloor \frac{3}{2}\Delta^*(H)\rfloor.$$
\end{lemma}

\begin{proof}
The triangular prism $D_3$ with three edges joining the two triangles replaced by disjoint paths of equal length 
$t+1, t \geq 0$, gives a graph $H$ with 
$\chi_c(H) = \Delta^*(H) + t +2 = \lfloor \frac{3}{2}\Delta^*(H)\rfloor.$ 
\end{proof}
\section{Reduction of Conjecture \ref{CCC:1}}

To prove Conjecture \ref{CCC:1} it is enough to prove the following:

\begin{conjecture}\label{CCC:2}
If $G$ is subdivision of a simple $3$-connected plane graph  with maximum face degree $\Delta^*(G)$, then   
$$\chi_c(G) \leq \lfloor \frac{3}{2}\Delta^*(G)\rfloor.$$
\end{conjecture}

\begin{theorem}\label{thm:3}
Conjecture \ref{CCC:1} holds if and only if Conjecture \ref{CCC:2} holds.
\end{theorem}

\begin{proof}
It is enough to prove that Conjecture \ref{CCC:1} follows from Conjecture \ref{CCC:2}. Suppose it is not true.

Let $H$ be a counterexample with minimum number of faces and then with the minimum number of vertices. It is easy to see that $H$ is $2$-connected, has at least four faces (because Lemma \ref{le:1}), and is not  a subdivision of any simple $3$-connected plane graph. Then, by Theorem \ref{thm:2}, $H$ contains a separating cycle $C_n$, with $n \leq \Delta^*(H)$. The graphs 
$H_1 = C_n \cup \mathrm {int}{_G(C_n)}$ and $H_2 = C_n \cup \mathrm {ext}{_G(C_n)}$, are smaller than $H$, therefore, for $i \in \{1, 2\}$ we have $$\chi_c(H_i) \leq \lfloor \frac{3}{2}\Delta^*(H_i)\rfloor.$$

Since the graphs $H_1$ and $H_2$ have only vertices of the cycle $C_n$ in common, $n \leq \Delta^*(H)$, and each face of $H$ is also a face in $H_1$ or in $H_2$, there is $\Delta^*(H_i) \leq \Delta^*(H)$. So we can combine the cyclic colorings of $H_1$ and $H_2$ to obtain a cyclic coloring of $H$ using at most $\frac{3}{2}\Delta^*(H)$ colors. A contradiction.
\end{proof}

\section{Cyclic colorings of subdivisions of simple $3$-connected plane graphs}

As in any simple $3$-connected plane graph every two distinct faces have at most one edge in common, for any subdivision $G$ of a simple $3$-connected plane graph we have $k^* = t+2, t = t(G)$. As a result, from Theorem \ref{thm:k^*}, we have:

\begin{theorem}\label{thm:4}
If $G$ is a subdivision of a simple $3$-connected plane graph with $\Delta^*(G) \geq 5$, then
$$\chi_c(G) \leq \max\{\Delta^*(G) + 3t + 8,\Delta^*(G) + 14\}.$$
\end{theorem}
 
\begin{corollary}\label{cor:1}
If $G$ is a subdivision of a simple $3$-connected plane graph with $\Delta^*(G) \geq \max\{6t + 16, 28\}$, then $$\chi_c(G) \leq \lfloor \frac{3}{2}\Delta^*(G)\rfloor.$$

\end{corollary}

\bigskip

Let $G$ be a $2$-connected plane graph with the reduction $R = R(G)$ being a simple $3$-connected plane graph. 
We associate with the graph $G$ the following plane multigraph $S = S(G)$, called the {\it{subdivision multigraph}} of $G$, whose vertex set $V(S) = F(G)$, the face set of $G$, and edge set $E(S) = V_2(G)$, the set of vertices of degree 2 of $G$.
The edge $v$ joins the vertices $f_a$ and $f_b$ in $S$ if and only if the 2-vertex $v$ is incident, in $G$, with both faces $f_a$ and $f_b$. Observe that the multigraph $S$ has the maximum vertex degree 
$\Delta(S) = \max_{f\in F(G)}\{ \deg_G(f) - \deg_R(f')\}$ where $f'$ is the face of $R$ corresponding to the face $f$ of $G$. Let 
$\chi'(S)$ denote the chromatic index of the subdivision multigraph of $S$.

\begin{theorem}\label{thm:6}
Let $G$ be a subdivision of a simple $3$-connected plane graph $R$ (which is the reduction $R = R(G)$) and let $S=S(G)$ be the subdivision multigraph of $G$. Then 
$$\chi_c(G) \leq \chi'(S) + \chi_c(R). $$
Moreover, the bound is tight.
\end{theorem}

\begin{proof}
Let $\phi: V(R) \rightarrow \{1,\dots, \chi_c(R)\}$ be a cyclic coloring of $R$. 
 
Let $\psi: V_2(G) = E(S)\rightarrow \{\chi_c(R) + 1,\dots, \chi_c(R) + \chi'(S)\}$ be a proper edge-coloring of $S$
with $\chi'(S)$ colors. The edge-coloring $\psi$ induces a vertex-coloring  of vertices of $V_2(G)$ in which every two distinct vertices $v_1$ and $v_2$ of degree 2, that are incident with the same face $f$, receive different colors. 
  It is easy to see that the colorings $\phi$ and $\psi$ together give a cyclic coloring of $G$.

The bound is tight for any $2$-connected plane graph in which any two its vertices are cyclically adjacent. For an example of such a graph $H$ consider the triangular prism $D_3$ with three edges joining two triangles replaced by three paths of lengths 
$a+1$, $b+1$, and $c+1$. It is easy to see that 
$$\chi_c(H) = (a + b + c) + 6  = \chi'(S(H)) + \chi_c(R(H)). $$
\end{proof}

\begin{theorem}\label{thm:7}
If $G$ is a subdivision of a $3$-connected plane graph $R$, then 
$$\chi_c(G) \leq \lfloor \frac32 \max_{f\in F(G)}\{ \deg_G(f) -\deg_R(f')\}\rfloor + \chi_c(R),$$

where $f'$ is the face of $R$ corresponding to the face $f$ of $G$. 
Moreover, the bound is tight.

\end{theorem}

\begin{proof}
By theorem of Shannon \cite{Sha-49} it holds that $\chi'(S) \leq \lfloor\frac32 \Delta(S)\rfloor.$
To see the tightness, consider the graph $H$ of the proof of Theorem \ref{thm:6} with $a = b = c \geq 2$.
\end{proof}

 Next corollary provides three other wide families of graphs for which Borodin's conjecture holds.

\begin{corollary}
If $G$ is a subdivision of a plane triangulation, a simple $3$-connected plane quadrangulation, or a simple $3$-connected plane pentagulation (all faces are of degree 5) with even maximum  face degree, then 

$$\chi_c(G) \leq \lfloor \frac32 \Delta^*(G)\rfloor.$$
\end{corollary}

\begin{proof}
We apply Theorem \ref{thm:7}. 
If $G$ is a triangulation, then $\chi_c(G) \leq \lfloor \frac{3}{2}(\Delta^*(G)- 3)\rfloor + 4 = \lfloor \frac{3\Delta^* - 9 + 8}{2}\rfloor \leq \lfloor\frac32 \Delta^*(G)\rfloor.$ 

If $G$ is a simple $3$-connected plane quadrangulation, then $\chi_c(G) \leq \lfloor \frac{3}{2}(\Delta^*(G)- 4)\rfloor + 6 = \lfloor \frac{3\Delta^* - 12 + 12}{2}\rfloor = \lfloor\frac32 \Delta^*(G)\rfloor.$

If $G$ is a subdivision of a simple $3$-connected plane pentagulation of even maximum face degree, then $\chi_c(G) \leq \lfloor \frac{3}{2}(\Delta^*(G)- 5)\rfloor + 8 = \lfloor \frac{3\Delta^* - 15 + 16}{2}\rfloor = \lfloor\frac32 \Delta^*(G)\rfloor.$
   
\end{proof}

\bigskip
\section{One more upper bound on cyclic chromatic number}

\begin{theorem}\label{thm:8}
If $G$ is a subdivision of a simple $3$-connected plane graph $R$, then 
$$\chi_c(G) \leq \max_{f\in F(G)}\{ \deg_G(f) -\deg_R(f')\} + t(G) + \chi_c(R),$$

where $f'$ is the face of $R$ corresponding to the face $f$ of $G$.
Moreover, the bound is tight.

\end{theorem}
 
\begin{proof}
We use for $\chi'(S)$ in the proof of Theorem \ref{thm:7}, instead of the Shannon bound,  the bound of Vizing and Gupta (see \cite{West-01}, p. 275), saying that $\chi'(S) \leq \Delta(S) + \mu(S)$ where 
$\mu(S)$ denotes the maximum edge multiplicity of $S$, and the fact that $\mu(S) = t(G)$. To see the tightness, consider the graph $H$ from the proof of Theorem \ref{thm:6} with $a = b = c = t(H) \geq 2$.
\end{proof}

\begin{corollary}\label{cor:2}
If $G$ is a subdivision of a simple $3$-connected plane graph $R$ with $\Delta^*(G) \geq 2\chi_c(R) + 2t(G) - 6$, then 
$$\chi_c(G) \leq \lfloor \frac{3}{2}\Delta^*(G)\rfloor.$$

\end{corollary}

Note that if $\Delta^*(R) \leq 2t(G) + 6$ for $t \geq 2$ or $\Delta^*(R) \leq 12-t$ for $t \leq 1$, then the bound of 
Theorem \ref{thm:8} is better than that of Theorem \ref{thm:4} (and, hence, of Theorem \ref{thm:k^*}).
\bigskip

The graph $G$ is a \textit{regular} subdivision of a graph $R$ if it is obtained from $R$ by replacing each edge of $R$ with a path of length $k+1$ for some constant $k \geq 0$.

\begin{corollary}\label{thm:9}
If $G$ is a regular subdivision of a simple $3$-connected plane graph $R$, then 
$$\chi_c(G) \leq \Delta^*(G) + t(G) + r,$$

where $r$ is a constant depending on the properties of $R$, see Theorem \ref{thm:+r}.
\end{corollary}

\begin{proof}
It is easy to see that for any face $f$ of $G$ there is $\deg_G(f) = \deg_R(f')(k +1)$ for some integer $k \geq 0$, where $f'$ is the face of $R$ corresponding to the face $f$. Then $t(G) = k$ and, by Theorem \ref{thm:8}, 
$\chi_c(G) \leq \max_{f\in F(G)}\{ \deg_G(f) -\deg_R(f')\} + t(G) + \Delta^*(R) + r = \max_{f\in F(G)}\{k\deg_R(f')\} + t(G) + \Delta^*(R) + r = \Delta^*(R)(k + 1) + t(G) + r = \Delta^*(G) + t(G) + r.$

\end{proof}

 
\begin{corollary}\label{thm:10}
If $G$ with $t(G) \geq 1$ is a regular subdivision of a simple $3$-connected plane graph $R$, then 
$$\chi_c(G) \leq \lfloor \frac{3}{2}\Delta^*(G)\rfloor.$$

\end{corollary}
\begin{proof}
Let $G$ be a regular subdivision  of the graph $R$ with $t(G) = k \geq 1$. From $\Delta^*(G) \geq 2 + \frac{8}{k+1}$ and Corollary \ref{thm:9}, we have $\chi_c(G) \leq \Delta^*(R)(k+1) +k + 5 \leq \lfloor\frac{3}{2}\Delta^*(R)(k+1)\rfloor = \lfloor\frac{3}{2}\Delta^*(G)\rfloor.$

\end{proof}

\bigskip
\section{One new conjecture}

We believe that the following conjecture, which involves both, the Plummer and Toft conjecture and the corresponding part of Conjecture 
\ref{BBC:1}, holds:

\begin{conjecture}\label{con:5}
If $G$ is a subdivision of a simple $3$-connected plane graph, then 
$$\chi_c(G) \leq \Delta^*(G) + t(G) + 2.$$
\end{conjecture}

Observe, that Conjecture \ref{con:5} holds for regular subdivisions of simple $3$-connected plane graphs $R$ with $\Delta^*(R) \geq 16$, or $\Delta^*(R) \leq 4$, or $\delta(R) = 4$ and $\Delta^*(R) \geq 6$, or $\delta(R) = 5$, or $R$ being locally connected (see Corollary \ref{thm:9} and Theorem \ref{thm:+r}, cases $r = 1$ and $r = 2$), for subdivisions of plane triangulations, and for subdivisions of simple $3$-connected plane quadrangulations.

In the case of triangulations we have from Theorem \ref{thm:8}:
$$\chi_c(G) \leq \Delta^*(G) - 3 + t(G) + 3 + 1 = \Delta^*(G) + t(G) + 1.$$

In the case of quadrangulations Theorem  \ref{thm:8} provides:
$$\chi_c(G) \leq \Delta^*(G) - 4 + t(G) + 4 + 2 = \Delta^*(G) + t(G) + 2.$$

From Lemma \ref{le:2} we know that the upper bound on $\chi_c(G)$ in our Conjecture \ref{con:5} cannot be improved.


\bigskip
{\bf Acknowledgments}

\medskip
This work was supported by the Slovak Research and Development Agency under the Contract No. APVV-19-0153.

\end{document}